\documentclass{article}

\usepackage{amsmath,amssymb}
\usepackage{amsthm}
\usepackage{mathrsfs,MnSymbol}
\usepackage{textcomp}
\usepackage{graphicx}
\usepackage{cite} 
\usepackage{enumerate}
\usepackage{kbordermatrix}
\usepackage{url}
\usepackage{chemarrow}

\usepackage{color}

\usepackage{harpoon}

\usepackage{tikz}
\usetikzlibrary{external}
\usetikzlibrary{positioning}
\usetikzlibrary{arrows}
\tikzset{
every picture/.style=thick,
bluenode/.style={circle, draw=black, fill=blue!40, very thick, minimum size=3mm,inner sep=1mm},
whitenode/.style={circle, draw=black, fill=black!10, very thick, minimum size=3mm,inner sep=1mm},
squarednode/.style={rectangle, draw=red!60, fill=red!5, very thick, minimum size=3mm},
every loop/.style={min distance=8mm} 
}

\newtheorem{theorem}{Theorem}[section]
\newtheorem{lemma}[theorem]{Lemma}
\newtheorem{proposition}[theorem]{Proposition}
\newtheorem{corollary}[theorem]{Corollary}

\newtheorem{definition}[theorem]{Definition}
\newtheorem{observation}[theorem]{Observation}
\newtheorem{remark}[theorem]{Remark}
\newtheorem{example}[theorem]{Example}
\newtheorem{question}[theorem]{Question}

\newenvironment{thm}{\begin{theorem}}{\end{theorem}}
\newenvironment{lem}{\begin{lemma}}{\end{lemma}}
\newenvironment{prop}{\begin{proposition}}{\end{proposition}}
\newenvironment{cor}{\begin{corollary}}{\end{corollary}}

\newenvironment{defn}{\begin{definition}\bgroup\rm }{\egroup\end{definition}}
\newenvironment{obs}{\begin{observation}\bgroup\rm }{\egroup\end{observation}}
\newenvironment{rem}{\begin{remark}\bgroup\rm }{\egroup\end{remark}}
\newenvironment{ex}{\begin{example}\bgroup\rm }{\egroup\end{example}}

\def \nul {\operatorname{null}}

\def \S {\mathcal{S}}

\def \lg#1#2#3#4{\phi_{#4}(#1,#2,#3)} 
\def \lset#1#2#3{\phi_{#3}(#1,#2)}
\def \ft#1#2#3{(#1:#2\rightarrow #3)}
\def \ol#1{\overline{#1}}

\def \Zsap{Z_{\mathrm{SAP}}}

\def \Zvc{Z_{\mathrm{vc}}}

\def \ZFloor{\lfloor Z\rfloor}
\def \dunion{\dot\cup}
\def \ccl{\ol{cl}}

\begin{document}

\title{Using a new zero forcing process to guarantee the Strong Arnold Property}

\author{
        Jephian C.-H. Lin\footnotemark[2]
        }

\date{\today}

\maketitle

\renewcommand{\thefootnote}{\fnsymbol{footnote}}
\footnotetext[2]{        
    Department of Mathematics, Iowa State University, 
    Ames, IA 50011, USA (chlin@iastate.edu).
		}
        
\renewcommand{\thefootnote}{\arabic{footnote}}

\begin{abstract}
The maximum nullity $M(G)$ and the Colin de Verdi\`ere type parameter $\xi(G)$ both consider the largest possible nullity over matrices in $\S(G)$, which is the family of real symmetric matrices whose $i,j$-entry, $i\neq j$, is nonzero if $i$ is adjacent to $j$, and zero otherwise;  however, $\xi(G)$ restricts to those matrices $A$ in $\S(G)$ with the Strong Arnold Property, which means $X=O$ is the only symmetric matrix that satisfies $A\circ X=O$, $I\circ X=O$, and $AX=O$.  This paper introduces zero forcing parameters $\Zsap(G)$ and $\Zvc(G)$, and proves that $\Zsap(G)=0$ implies every matrix $A\in \S(G)$ has the Strong Arnold Property and that the inequality $M(G)-\Zvc(G)\leq \xi(G)$ holds for every graph $G$.  Finally, the values of $\xi(G)$ are computed for all graphs up to $7$ vertices, establishing  $\xi(G)=\ZFloor(G)$ for these graphs.
\end{abstract}

\noindent{\bf Keywords:} 
Strong Arnold Property, SAP zero forcing, minimum rank, maximum nullity, Colin de Verdi\`ere type parameter, vertex cover.
\medskip 

\noindent{\bf AMS subject classifications:}
05C50, 
05C57, 
05C83, 
15A03, 
15A18, 
15A29. 

\section{Introduction}
A \textit{minimum rank problem} for a graph $G$ is to determine what is the smallest possible rank, or equivalently the largest possible nullity, among a family of matrices associated with $G$.  One classical way to associate matrices to a graph $G$ is through $\S(G)$, which is defined as the set of all real symmetric matrices whose $i,j$-entry, $i\neq j$, is nonzero whenever $i$ and $j$ are adjacent in $G$, and zero otherwise.  Note that the diagonal entries can be any real number.  Another association is $\S_+(G)$, which is the set of positive semidefinite matrices in $\S(G)$.  Thus, the \textit{maximum nullity} $M(G)$ and the \textit{positive semidefinite maximum nullity} $M_+(G)$ are defined as 
\begin{align*}
M(G)&=\max\{\nul(A):A\in\S(G)\}, \text{ and}\\
M_+(G)&=\max\{\nul(A):A\in\S_+(G)\}.\\
\end{align*}

The classical minimum rank problem is a branch of the \textit{inverse eigenvalue problem}, which asks for a given multi-set of real numbers, is there a matrix in $\S(G)$ such that its spectrum is composed of these real numbers.  If $\lambda$ is an eigenvalue of some matrix $A\in \S(G)$, then its multiplicity should be no higher than $M(G)$, for otherwise $A-\lambda I$ has nullity higher than $M(G)$.  Similarly, $M_+(G)$ provides an upper bound for the multiplicities of the smallest and the largest eigenvalues.  Also, $M_+(G)$ is closely related to faithful \textit{orthogonal representations} \cite{HLA46}.

Other families of matrices are defined through the Strong Arnold Property.  A matrix $A$ is said to have the \textit{Strong Arnold Property} (or SAP) if the zero matrix is the only symmetric matrix $X$ that satisfies the three conditions $A\circ X=O$, $I\circ X=O$, and $AX=O$.  Here $I$ and $O$ are the identity matrix and the zero matrix of the same size as $A$, respectively, and $\circ$ is the Hadamard (entrywise) product of matrices.  By adding the SAP to the conditions of the abovementioned families, the \textit{Colin de Verdi\`ere type parameters} are defined as 
\begin{align*}
\xi(G)&=\max\{\nul(A):A\in\S(G), A\text{ has the SAP}\}\text{ \cite{xi}, and}\\
\nu(G)&=\max\{\nul(A):A\in\S_+(G), A\text{ has the SAP}\}\text{ \cite{CdV2}.}
\end{align*}

These parameters are variations of the original Colin de Verdi\`ere parameter $\mu(G)$ \cite{CdV}, which is defined as the maximum nullity over matrices $A\in\S(G)$ such that 
\begin{itemize}
\item every off-diagonal entry of $A$ is non-positive (called a \textit{generalized Laplacian}),
\item $A$ has exactly one negative eigenvalue including the  multiplicity, and 
\item $A$ has the SAP.
\end{itemize}
In order to see how the SAP makes a difference between these parameters, we define $M_\mu(G)$ as the maximum nullity of the same family of matrices by ignoring the SAP, i.e.~the maximum nullity of matrices $A\in \S(G)$ such that $A$ is a generalized Laplacian and has exactly one negative eigenvalue.

The SAP gives $\xi(G)$, $\nu(G)$, and $\mu(G)$ nice properties.  For example, they are \textit{minor monotone} \cite{HLA46}.  A graph $H$ is a \textit{minor} of a graph $G$ if $H$ can be obtained from $G$ by a sequence of deleting edges, deleting vertices, and contracting edges; a graph parameter $\zeta$ is said to be minor monotone if $\zeta(H)\leq\zeta(G)$ whenever $H$ is a minor of $G$.  By the graph minor theorem (e.g.,~see \cite{DiestelGT}), for a given integer $d$ and a minor monotone parameter $\zeta$, the minimal forbidden minors for $\zeta(G)\leq d$ consist of only finitely many graphs.  Here $\zeta$ can be $\xi$, $\nu$ or $\mu$.  More specifically, $\mu(G)\leq 3$ if and only if $G$ is a planar graph \cite{KLV}, which is characterized by the forbidden minors $K_5$ and $K_{3,3}$.  

However, the SAP also makes the Colin de Verdi\`ere type parameters less controllable by the existing tools.  For example, zero forcing parameters, which will be defined in Section \ref{subsec:zeroforcing}, were used extensively as a bound for the minimum rank problem.  For the classical zero forcing number $Z(G)$, it is known that $M(G)\leq Z(G)$ for all graphs \cite{AIM}; and $M(G)=Z(G)$ when $G$ is a tree or $|G|\leq 7$ \cite{AIM,small}.  An analogy for $\xi(G)$ is the minor monotone floor of the zero forcing number, which is denoted as $\ZFloor(G)$ and will be defined in Section \ref{sec:xi}.  It is known that $\xi(G)\leq \ZFloor(G)$ for all graphs \cite{param}.  The similar statement $\xi(G)=\ZFloor(G)$ is not always true when $G$ is a tree \cite{param}, and no results about $\xi(G)$ and $\ZFloor(G)$ for small graphs are known. 

The main goal of this paper is to establish a connection between zero forcing parameters and the SAP, and derive consequences.  This leads to some questions.  Does some graph structure guarantee that every $A\in \S(G)$ has the SAP?  Thus, the maximum nullity does not change when the SAP condition is added.  Or, is there a strategy to perturb any given matrix such that it guarantees the SAP?  Thus, the rank changed by the perturbation gives an upper bound for $M(G)-\xi(G)$.

In Section \ref{sec:Zsap}, we introduce a new parameter $\Zsap(G)$ and its variants $\Zsap^\ell$ and $\Zsap^+$, and prove in Theorem \ref{thm:Zsap} that under the condition $\Zsap(G)=0$,  every matrix $A\in \S(G)$ has the SAP.  Thus, $\xi(G)=M(G)$, $\nu(G)=M_+(G)$, and $\mu(G)=M_\mu(G)$ when $\Zsap(G)=0$, so finding the value of Colin de Verdi\`ere type parameters is equivalent to finding the value of the corresponding parameters.  Table \ref{tbl:Zsap} in Section \ref{sec:simulation} indicates that there are actually a considerable proportion of graphs that have this property.

In Section \ref{sec:Zvc}, another parameter $\Zvc(G)$ and its variant $\Zvc^\ell(G)$ are defined, and Theorem \ref{thm:Zvc} states that $M(G)-\xi(G)\leq \Zvc(G)$ for every graph $G$.  With the help of $\Zsap(G)$, $\Zvc(G)$, and some existing theorems, Section \ref{sec:xi} provides the result that $\xi(G)=\ZFloor(G)$ for graphs $G$ up to $7$ vertices.

All parameters introduced in this paper and their relations are illustrated in Figure \ref{fig:param}.  A brief description of the related theorems are given on the sides.  A line between two parameters means the lower one is less than or equal to the upper one.

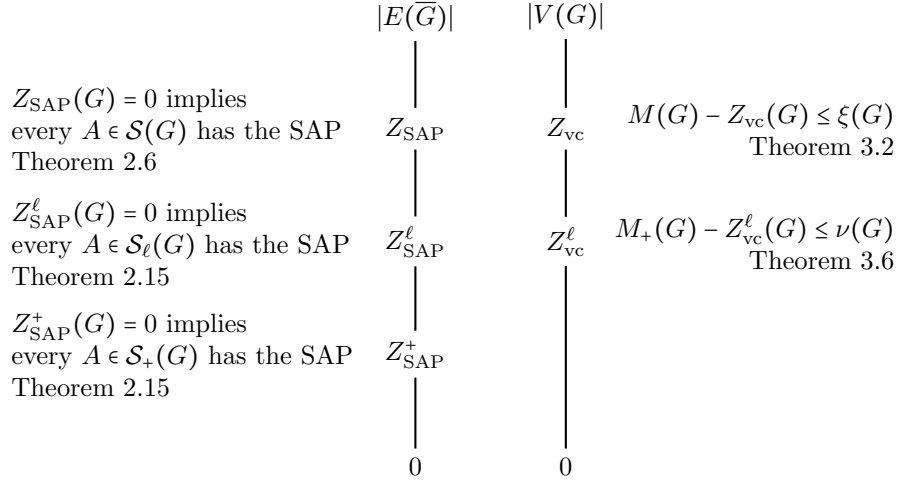
\begin{figure}[h]
\begin{center}\begin{tikzpicture}
\foreach \i in {0,1}{
\pgfmathsetmacro{\x}{2*(\i-0.5)}
    \foreach \j in {1,...,5}{
    \pgfmathsetmacro{\y}{-1.5*\j}
    \path (\x,\y) coordinate (c\i\j);
    }
}
\node (v01) at (c01) {$|E(\ol{G})|$};
\node (v02) at (c02) {$\Zsap$};
\node (v03) at (c03) {$\Zsap^\ell$};
\node (v04) at (c04) {$\Zsap^+$};
\node (v05) at (c05) {$0$};
\draw (v01) -- (v02) -- (v03) -- (v04) -- (v05);

\node (v11) at (c11) {$|V(G)|$};
\node (v12) at (c12) {$\Zvc$};
\node (v13) at (c13) {$\Zvc^\ell$};
\node (v14) at (c14) {};
\node (v15) at (c15) {$0$};
\draw (v11) -- (v12) -- (v13) --(v15);

\begin{scope}[every node/.style={align=left, xshift=-5.5cm,right}]
\node at (c02) {$\Zsap(G)=0$ implies\\ every $A\in \S(G)$ has the SAP\\ Theorem \ref{thm:Zsap}};
\node at (c03) {$\Zsap^\ell(G)=0$ implies\\ every $A\in \S_\ell(G)$ has the SAP\\ Theorem \ref{thm:Zsapl+}};
\node at (c04) {$\Zsap^+(G)=0$ implies\\ every $A\in \S_+(G)$ has the SAP\\ Theorem \ref{thm:Zsapl+}};
\end{scope}

\begin{scope}[every node/.style={align=right, xshift=4.5cm, left}]
\node at (c12) {$M(G)-\Zvc(G)\leq \xi(G)$\\ Theorem \ref{thm:Zvc}};
\node at (c13) {$M_+(G)-\Zvc^\ell(G)\leq \nu(G)$\\ Theorem \ref{thm:Zvcl}};
\end{scope}

\end{tikzpicture}\end{center}
\caption{Parameters introduced in this paper.}
\label{fig:param}
\end{figure}

  Throughout the paper, the neighborhood of a vertex $i$ in a graph $G$ is denoted as $N_G(i)$, while the closed neighborhood is denoted as $N_G[i]$, which equals $N_G(i)\cup \{i\}$.  The induced subgraph on a vertex set $W$ of $G$ is denoted as $G[W]$.  If $A$ is a matrix, $U$ and $W$ are subsets of the row and column indices of $A$ respectively, then $A[U,W]$ is the submatrix of $A$ induced on the rows of $U$ and columns of $W$; if $U$ and $W$ are ordered sets, then permute the rows and columns of this submatrix accordingly.
  
\subsection{SAP system and its matrix representation}
Let $G$ be a graph on $n$ vertices, and $\ol{m}=|E(\ol{G})|$.  In order to see if a matrix $A\in \S(G)$ has the SAP or not, the matrix $X$ can be viewed as a symmetric matrix with $\ol{m}$ variables at the positions of non-edges so that $X$ satisfies $A\circ X=I\circ X=O$.  Next, $AX=O$ leads to $n^2$ restrictions on the $\ol{m}$ variables, which forms a linear system.  
Call this linear system the \textit{SAP system of $A$}, which can also be written as an $n^2\times \ol{m}$ matrix.

\begin{defn}
\label{defn:system}
Let $G$ be a graph on $n$ vertices, $\ol{m}=|E(\ol{G})|$, and $A=\begin{bmatrix}a_{i,j}\end{bmatrix}\in \S(G)$.  Given an order of the set of non-edges, the \textit{SAP matrix} of $A$ with respect to this order is an $n^2\times \ol{m}$ matrix $\Psi{}$ whose rows are indexed by pairs $(i,k)$ and columns are indexed by the non-edges $\{j,h\}$ such that 
\[\Psi{}_{(i,k),\{j,h\}}=\left\{\begin{array}{ll}
0 & \text{if }k\notin\{j,h\},\\
a_{i,j} & \text{if }k\in\{j,h\} \text{ and }k=h.\\
\end{array}\right.\]
The rows follow the order $(i,k)< (j,h)$ if and only if $k< h$, or $k=h$ and $i< j$;  the columns follow the order of the non-edges.  
\end{defn}

\begin{rem}
\label{rem:system}
Let $G$ be a graph, $A\in \S(G)$, and $\Psi$ the SAP matrix of $A$ with respect a given order of the non-edges.  The columns of $\Psi$ correspond to the $\ol{m}$ variables in $X$, and the row for $(i,j)$ represents the equation $(AX)_{i,j}=0$.  Therefore, a matrix has the SAP if and only if the corresponding SAP matrix is full-rank.  

The rows of $\Psi{}$ can be partitioned into $n$ blocks, each having $n$ elements.  The $k$-th block are those rows indexed by $(i,k)$ for $1\leq i\leq n$.  Let ${\bf v}_j$ be the $j$-th column of $A$.  For the submatrix of $\Psi{}$ induced by the rows in the $k$-th block, the $\{j,h\}$ column is ${\bf v}_j$ if $k\in \{j,h\}$ and $k=h$, and is a zero vector otherwise.  Equivalently, on the $\{i,j\}$ column of $\Psi{}$, the $i$-th block is ${\bf v}_j$, the $j$-th block is ${\bf v}_i$, while other blocks are zero vectors.
\end{rem}

\begin{ex}
\label{ex:systems}
Let $G=P_4$ be the path on four vertices, labeled by the linear order.  Consider a matrix $A\in\S(G)$ and the matrix $X$ with three variables, as shown below.  
\[AX=\kbordermatrix{&1&2&3&4\\
1&-1 & 1 & 0 & 0\\
2&1 & -1 & 1 & 0\\
3&0 & 1 & -1 & 1\\
4&0 & 0 & 1  & -1\\
}\kbordermatrix{&1&2&3&4\\
&0 & 0 & x_{\{1,3\}} & x_{\{1,4\}}\\
&0 & 0 & 0 & x_{\{2,4\}}\\
&x_{\{1,3\}} & 0 & 0 & 0\\
&x_{\{1,4\}} & x_{\{2,4\}} & 0  & 0\\
}\]

The SAP matrix of $A$ with respect to the order $(\{1,3\},\{1,4\},\{2,4\})$ is a matrix $\Psi$ representing the linear system for $AX=O$ with three variables $x_{\{1,3\}},x_{\{1,4\}},x_{\{2,4\}}$.  
For convenience, write $A=\begin{bmatrix}{\bf v}_1&{\bf v}_2&{\bf v}_3&{\bf v}_4\end{bmatrix}$, where ${\bf v}_j$ is the $j$-th column vector of $A$.  Now $AX=O$ means 
\[\sum_{j\notin N_G[k]}x_{\{j,k\}}{\bf v}_j={\bf 0}\text{ for each }k\in V(G).\]
Thus,  
\[
\Psi =
\kbordermatrix{
& x_{\{1,3\}} &x_{\{1,4\}}&x_{\{2,4\}}\\
1&{\bf v}_3 & {\bf v}_4& {\bf 0}\\
2&{\bf 0}   & {\bf 0}  & {\bf v}_4\\
3&{\bf v}_1 & {\bf 0}  & {\bf 0}\\
4&{\bf 0}   & {\bf v}_1& {\bf v}_2\\
}=
\kbordermatrix{
& x_{\{1,3\}} &x_{\{1,4\}}&x_{\{2,4\}}\\
(1,1) &0 & 0 & 0 \\
(2,1) &1 & 0 & 0 \\
(3,1) &-1 & 1 & 0 \\
(4,1) &1 & -1 & 0 \\
(1,2) &0 & 0 & 0 \\
(2,2) &0 & 0 & 0 \\
(3,2) &0 & 0 & 1 \\
(4,2) &0 & 0 & -1 \\
(1,3) &-1 & 0 & 0 \\
(2,3) &1 & 0 & 0 \\
(3,3) &0 & 0 & 0 \\
(4,3) &0 & 0 & 0 \\
(1,4) &0 & -1 & 1 \\
(2,4) &0 & 1 & -1 \\
(3,4) &0 & 0 & 1 \\
(4,4) &0 & 0 & 0
}.
\]
\end{ex} 

\subsection{Zero forcing parameters}
\label{subsec:zeroforcing}

On a graph $G$, the conventional \textit{zero forcing game} (ZFG) is a color-change game such that each vertex is colored blue or white initially, and then the \textit{color change rule} (CCR) is applied repeatedly.  If starting with an initial blue set $B\subseteq V(G)$ and every vertex turns blue eventually, this set $B$ is called a \textit{zero forcing set} (ZFS).  The zero forcing number is defined as the minimum cardinality of a ZFS.

Different types of zero forcing numbers are discussed in the literature (e.g.,~see \cite{HLA46,smallparam,param}).  
Most of them serve as upper bounds of different types of maximum nullities.  Here we consider three types of the zero forcing numbers $Z$, $Z_\ell$, $Z_+$ with the corresponding color change rules:
\begin{itemize}
\item  (CCR-$Z$) If $i$ is a blue vertex and $j$ is the only white neighbor of $i$, then $j$ turns blue.
\item  (CCR-$Z_\ell$) CCR-$Z$ can be used to perform a force.  Or if $i$ is a white vertex without white neighbors and $i$ is not isolated, then $i$ turns blue.
\item  (CCR-$Z_+$) Let $B$ be the set of blue vertices at some stage and $W$ the vertices of a component of $G-B$.  CCR-$Z$ is applied to $G[B\cup W]$ with blue vertices $B$.
\end{itemize}
When a zero forcing game is mentioned, it is equipped with a color change rule, and we use $i\rightarrow j$ to denote a corresponding force (i.e.~$i$ forcing $j$ to become blue).  Note that for CCR-$Z_\ell$, it is possible to have $i\rightarrow i$.

It is known \cite{AIM,smallparam,param} that $M(G)\leq Z(G)$, $M_+(G)\leq Z_+(G)$, and $Z_+(G)\leq Z_\ell(G)\leq Z(G)$.  Denote $\S_\ell(G)$ as those matrices in $\S(G)$ whose $i,i$-entry is zero if and only if vertex $i$ is an isolated vertex.  Then every matrix $A\in \S_\ell(G)$ has nullity at most $Z_\ell(G)$ \cite{cancun}.  

All these results rely on Proposition \ref{prop:forcing}.

\begin{prop}{\rm \cite{AIM,smallparam,cancun}}
\label{prop:forcing}
Let $G$ be a graph on $n$ vertices.  Suppose at some stage $B$ is the set of blue vertices.  
\begin{itemize}
\item If $i\rightarrow j$ under CCR-$Z$, then for any matrix $A\in \S(G)$ with column vectors $\{{\bf v}_s\}_{s=1}^n$, $\sum_{s\notin B} x_s{\bf v}_s={\bf 0}$ implies $x_j=0$.  
\item If $i\rightarrow j$ under CCR-$Z_\ell$, then for any matrix $A\in\S_\ell(G)$ with column vectors $\{{\bf v}_s\}_{s=1}^n$, $\sum_{s\notin B} x_s{\bf v}_s={\bf 0}$ implies $x_j=0$.
\item If $i\rightarrow j$ under CCR-$Z_+$, then for any matrix $A\in\S_+(G)$ with column vectors $\{{\bf v}_s\}_{s=1}^n$, $\sum_{s\notin B} x_s{\bf v}_s={\bf 0}$ implies $x_j=0$.
\end{itemize}
\end{prop}

\section{SAP zero forcing parameters}
\label{sec:Zsap}

In this section, we introduce a new parameter $\Zsap(G)$ and prove that if $\Zsap(G)=0$ then every matrix $A\in\S(G)$ has the SAP, which implies $M(G)=\xi(G)$.  We also introduce similar parameters and results for other variants.  

First we give two examples illustrating what we called in Definition \ref{defn:Zsap} the forcing triple and the odd cycle rule.

\begin{ex}
\label{ex:P4}
Consider the graph $P_4$.  Let $A$ be the matrix as in Example \ref{ex:systems} and ${\bf v}_j$ its $j$-th column.  In Example \ref{ex:systems}, we know the SAP matrix of $A$ can be written as 
\[
\kbordermatrix{
& x_{\{1,3\}} &x_{\{1,4\}}&x_{\{2,4\}}\\
1&{\bf v}_3 & {\bf v}_4& {\bf 0}\\
2&{\bf 0}   & {\bf 0}  & {\bf v}_4\\
3&{\bf v}_1 & {\bf 0}  & {\bf 0}\\
4&{\bf 0}   & {\bf v}_1& {\bf v}_2\\
}.\]
Since ${\bf v}_4$ is the only nonzero vector on the second block-row, $x_{\{2,4\}}$ must be $0$ in this linear system.  Similarly, ${\bf v}_1$ is the only nonzero vector on the third block-row, so $x_{\{1,3\}}=0$.  Provided that $x_{\{1,3\}}=x_{\{2,4\}}=0$, the structure on the first block-row forces $x_{\{1,4\}}=0$.  Since this argument holds for every matrix in $\S(G)$, every matrix in $\S(G)$ has the SAP.
\end{ex}

\begin{ex}
\label{ex:K13}
Let $G=K_{1,3}$.  Consider the matrices $A$ and $X$ as 
\[A=\begin{bmatrix}
d_1 & a_1 & a_2 & a_3 \\
a_1 & d_2 & 0 & 0 \\
a_2 & 0 & d_3 & 0 \\
a_3 & 0 & 0 & d_4 \\
\end{bmatrix}\text{ and }X=
\begin{bmatrix}
0 & 0 & 0 & 0 \\
0 & 0 & x_{\{2,3\}} & x_{\{2,4\}} \\
0 & x_{\{2,3\}} & 0 & x_{\{3,4\}} \\
0 & x_{\{2,4\}} & x_{\{3,4\}} & 0 \\
\end{bmatrix}.\]
Let ${\bf v}_j$ be the $j$-th column of $A$.  Then the SAP matrix of $A$ with respect to the order $(\{2,3\},\{3,4\},\{2,3\})$ can be written as 
\[
\Psi =
\kbordermatrix{
& x_{\{2,3\}} &x_{\{3,4\}}&x_{\{2,4\}}\\
1&{\bf 0}   & {\bf 0}  & {\bf 0}\\
2&{\bf v}_3 & {\bf 0}  & {\bf v}_4\\
3&{\bf v}_2 & {\bf v}_4& {\bf 0}\\
4&{\bf 0}   & {\bf v}_3& {\bf v}_2\\
}.\]
Recall that the row with index $(i,j)$ is the $i$-th row in the $j$-th block.  Thus the submatrix induced by rows $\{(1,2),(1,3),(1,4)\}$ is
\[\begin{bmatrix}
a_2 & 0 & a_3 \\
a_1 & a_3 & 0 \\
0 & a_2 & a_1 \\
\end{bmatrix},\]
whose determinant is always nonzero if $a_1,a_2,a_3\neq 0$.  This means the SAP matrix of $A$ is always full-rank, regardless the choice of $A\in\S(G)$.  Hence every matrix $A\in \S(G)$ has the SAP.  This reason behind this is because a $3$-cycle appears in $\ol{G}$.
\end{ex}

As shown in Example \ref{ex:P4} and Example \ref{ex:K13}, some graph structures guarantee that every matrix described by the graph has the SAP.  This assurance is given by forcing $x_e=0$ step by step or by the occurrence of some odd cycle inside $\ol{G}$.  Utilizing these ideas, we design the \textit{SAP zero forcing game}, where the information $x_e=0$ is stored by coloring the non-edge $e$ blue.

Different from the conventional zero forcing game, the SAP zero forcing game is coloring ``non-edges'' to be blue or white, instead of coloring vertices;  also, a set of initial blue non-edges is called a zero forcing set if every non-edge turns blue eventually by repeated applications of the given color change rules.

Let $G$ be a graph and $i\in V(G)$.  Recall that $N_G(i)$ is the neighborhood of $i$ in $G$.  For $B_E$ a set of edges (2-sets), the notation $N_{B_E}(i)$ denotes the vertices $j$ with $\{i,j\}\in B_E$.  

The definition of $\Zsap(G)$ uses the concept of local games, which we now define.  

\begin{defn}
\label{defn:lg}
Let $G$ be a graph with some non-edges $B_E$ colored blue, and $k\in V(G)$.  The \textit{local game} $\lg{G}{B_E}{k}{Z}$ is the conventional zero forcing game on $G$ equipped with CCR-$Z$ and the initial blue set $\lset{G}{B_E}{k}:=N_G[k]\cup N_{B_E}(k)$.  When $Z$ is replaced by another zero forcing rules, such as $Z_\ell$ or $Z_+$, the setting remains the same but a different rule applies.
\end{defn}

\begin{defn}
\label{defn:Zsap}
For a graph $G$, the \textit{SAP zero forcing number} $\Zsap(G)$ is the minimum number of blue non-edges such that every non-edge will become blue by repeated applications of the \textit{color change rule for} $\Zsap$ (CCR-$\Zsap$):
\begin{itemize}
\item Suppose at some stage, $B_E$ is the set of blue non-edges and $\{j,k\}$ is a white non-edge.  If $i\rightarrow j$ in $\lg{G}{B_E}{k}{Z}$ for some vertex $i$, then the non-edge $\{j,k\}$ is changed to blue.  This is denoted as $\ft{k}{i}{j}$.
\item Let $\ol{G}_W$ be the graph whose edges are the white non-edges.  If for some vertex $i$, $\ol{G}_W[N_G(i)]$ contains a component that is an odd cycle $C$, then all non-edges on $C$ turn blue.  This is denoted as $(i\rightarrow C)$.
\end{itemize}
The three vertices $i$, $j$, and $k$ in the first rule is called a \textit{forcing triple}; the second rule is called the \textit{odd cycle rule}.
\end{defn}

The odd cycle rule follows a similar idea from the odd cycle zero forcing number \cite{Zoc}.

\begin{lem}
\label{lem:oc}
For any nonzero real numbers $a_1,a_2,\ldots ,a_n$ with $n$ odd, a matrix of the form 
\[\begin{bmatrix}
a_2 & 0 & \cdots & 0 & a_n \\
a_1 & a_3 & 0 &  & 0\\
0 & a_2 & \ddots & \ddots &  \vdots \\
\vdots & 0 &\ddots & a_n & 0\\
0 & \cdots & 0 & a_{n-1} & a_1\\
\end{bmatrix}\]
is nonsingular.
\end{lem}
\begin{proof}
Let $A$ be a matrix of the described form.  When $n$ is odd, 
\[\det(A)=2\prod_{i=1}^na_i,\]
which is nonzero provided that $a_i$'s are all nonzero.  Hence $A$ is nonsingular.
\end{proof}

\begin{theorem}
\label{thm:Zsap}
Suppose $G$ is a graph with $\Zsap(G)=0$.  Then every matrix in $\S(G)$ has the SAP.  Therefore, $M(G)=\xi(G)$, $M_+(G)=\nu(G)$, and $M_\mu(G)=\mu(G)$.
\end{theorem}
\begin{proof}
Let $A=\begin{bmatrix}a_{i,j}\end{bmatrix}\in \S(G)$ with ${\bf v}_j$ as the $j$-th column vector.  Pick an order for the set of non-edges, and let $\Psi$ be the SAP matrix for $A$ with respect to the given order.  Suppose ${\bf x}$ is a vector such that $\Psi{}{\bf x}={\bf 0}$.  Then ${\bf x}=(x_e)_{e\in E(\ol{G})}$ such that the entries of ${\bf x}$ are indexed by the non-edges of $G$ in the given order.  We relate the SAP zero forcing game to the zero-nonzero pattern of ${\bf x}$.

{\bf Claim 1:} Suppose at some stage, $B_E$ is the set of blue non-edges, and $\ft{k}{i}{j}$ is a forcing triple.  Then $x_{e}=0$ for all $e\in B_E$ implies $x_{\{j,k\}}=0$.

To establish the claim, recall that the condition $\Psi{}{\bf x}={\bf 0}$ on those rows in the $k$-th block means
\[\sum_{s\notin N_G[k]}x_{\{s,k\}}{\bf v}_s={\bf 0}.\]
Suppose $x_{e}=0$ for all $e\in B_E$.  Then this equality reduces to 
\[\sum_{s\notin N_G[k]\cup N_{B_E}(k)}x_{\{s,k\}}{\bf v}_s={\bf 0}.\]
Since by Definition \ref{defn:lg} the set $\lset{G}{B_E}{k}=N_G[k]\cup N_{B_E}(k)$ is exactly the set of initial blue vertices in $\lg{G}{B_E}{k}{Z}$, the force $i\rightarrow j$ in $\lg{G}{B_E}{k}{Z}$ implies $x_{\{j,k\}}=0$ by Proposition \ref{prop:forcing}.

{\bf Claim 2:} Suppose at some stage, $B_E$ is the set of blue non-edges, and $(i\rightarrow C)$ is applied by the odd cycle rule.  Then $x_{e}=0$ for all $e\in B_E$ implies $x_{e}=0$ for every $e\in E(C)$. 

To establish the claim, let $\ol{G}_W$ be the graph whose edges are the white non-edges at this stage.  Since $(i\rightarrow C)$ is applied by the odd cycle rule, $C$ is a component in $\ol{G}_W[N_G(i)]$ and $|V(C)|=d$ is an odd number.  Following the cyclic order, write the vertices in $V(C)$ as $\{k_s\}_{s=1}^d$, and $e_s=\{k_s,k_{s+1}\}$, with the index taken modulo $d$. 
  
Denote $U=\{(i,k_s)\}_{s=1}^d$, $W_1=\{e_s\}_{s=1}^d$, and $W_2$ as those white non-edges not in $W_1$.  For each $(i,k_s)\in U$, $\Psi_{(i,k_s),e_{s-1}}=a_{i,k_{s-1}}$ and $\Psi_{(i,k_s),e_s}=a_{i,k_{s+1}}$; for every white non-edge $e=\{j,h\}$ other than $e_{s-1}$ and $e_s$, either $k_s\notin\{j,h\}$ or $k_s=h$ but $j$ is not adjacent to $i$, so $\Psi_{(i,k_s),e}=0$ by Definition \ref{defn:system}.  This means $\Psi[U,W_2]=O$ and $\Psi[U,W_1]$ is of the form described in Lemma \ref{lem:oc}.  Consequently, $x_{e}=0$ for all $e\in B_E$ implies $x_{e}=0$ for every non-edge $e\in E(C)$.

By the claims, $\Zsap(G)=0$ means all of the $x_{e}$ will be forced to zero, so ${\bf x}={\bf 0}$ is the only vector in the right kernel of $\Psi$.  This means $\Psi{}$ is full-rank.

Since the argument works for every matrix $A\in \S(G)$, $\Zsap(G)=0$ implies every matrix $A\in\S(G)$ has the SAP.  Consequently, $M(G)=\xi(G)$, $M_+(G)=\nu(G)$, and $M_\mu(G)=\mu(G)$.
\end{proof}

\begin{rem}
With and without the restriction of having the SAP, the inertia sets that can be achieved by matrices in $\S(G)$ are considered in the literature (e.g.,~see \cite{BHL09,AHLvdH13}).  With the help of Theorem \ref{thm:Zsap}, if $\Zsap(G)=0$, then these two inertia sets are the same.
\end{rem}

\begin{cor}
\label{cor:forest}
If $G$ has no isolated vertices and $\ol{G}$ is a forest, then $\Zsap(G)=0$ and every matrix in $\S(G)$ has the SAP.
\end{cor}
\begin{proof}
Suppose at some stage $\ol{G}_W$ is the graph whose edges are the white non-edges.  Since $\ol{G}$ is a forest, $\ol{G}_W$ always has a leaf $k$, unless $\ol{G}_W$ contains no edge.  Let $j$ be the only neighbor of $k$ in $\ol{G}_W$, and let $i$ be one of the neighbor of $j$ in $G$.  Since $G$ has no isolated vertices, $i$ always exists.  Thus, in the local game $\lg{G}{E(\ol{G})\setminus E(\ol{G}_W)}{k}{Z}$, every vertex is blue except $j$, so $i\rightarrow j$.  Therefore, $\ft{k}{i}{j}$ applies and $\{j,k\}$ turns blue.  Continuing this process, all non-edge becomes blue, so $\Zsap(G)=0$.  
\end{proof}

Note that the condition that $G$ has no isolated vertices is crucial for Corollary \ref{cor:forest}.  For example, $\Zsap(\ol{K_{1,n}})> 0$.  In fact, $\Zsap(G)=0$ does not happen only when $\ol{G}$ is a forest.  Example \ref{ex:Zsap0} gives a graph $G$ such that $\ol{G}$ is not a forest and $\Zsap(G)=0$.  We will see in Table \ref{tbl:Zsap} that there are a considerable number of graphs having the property $\Zsap(G)=0$.

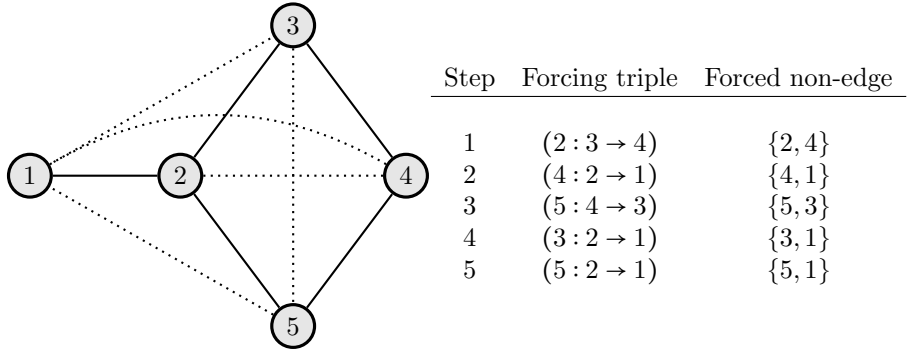
\begin{figure}[h]
\begin{center}\begin{tikzpicture}
\node [whitenode] (1) at (-2.5,0) {1};
\node [whitenode] (2) at (-0.5,0) {2};
\node [whitenode] (3) at (1,2) {3};
\node [whitenode] (4) at (2.5,0) {4};
\node [whitenode] (5) at (1,-2) {5};

\draw (1)--(2)--(3)--(4)--(5)--(2);
\path[dotted] (2) edge (4)
    (4) edge [bend right] (1)
    (1) edge (3)
    (3) edge (5)
    (5) edge (1);
    
\node[anchor=west] at (2.7,0) {
$\begin{array}{ccc}
\text{Step}& \text{Forcing triple} & \text{Forced non-edge}\\
\hline \\
1&\ft{2}{3}{4}&\{2,4\}\\
2&\ft{4}{2}{1}&\{4,1\}\\
3&\ft{5}{4}{3}&\{5,3\}\\
4&\ft{3}{2}{1}&\{3,1\}\\
5&\ft{5}{2}{1}&\{5,1\}\\
\end{array}$
};

\end{tikzpicture}\end{center}
\caption{The graph $G$ for Example \ref{ex:Zsap0} and the forcing process.}
\label{fig:kite}
\end{figure}

\begin{ex}
\label{ex:Zsap0}
Let $G$ be the graph shown in Figure \ref{fig:kite}.  Following the steps listed in Figure \ref{fig:kite}, every non-edge turns blue, so $\Zsap(G)=0$.  Observe that at the beginning, the graph $\ol{G}_W$ of white non-edges is the same as $\ol{G}$, and $\ol{G}_W[N_G(2)]$ is a $3$-cycle $C$, so one can also use the odd cycle rule to perform $(2\rightarrow C)$.  This will accelerate the process but not change the result.  By Theorem \ref{thm:Zsap}, every matrix $A\in\S(G)$ has the SAP, so $\xi(G)=M(G)$.  Since the number of vertices is no more than 7, $M(G)=Z(G)=2$ and thus $\xi(G)=2$.
\end{ex}

\begin{cor}
Let $G$ be any graph with diameter $2$ and $\Delta(G)\leq 3$.  Then $\Zsap(G)=0$.  In particular, when $G$ is the Petersen graph, $\Zsap(G)=0$, so $\xi(G)=M(G)=5$.
\end{cor}
\begin{proof}
For every white non-edge $\{j,k\}$, there is at least one common neighbor $i$ of $j$ and $k$, since the diameter is $2$.  By the assumption, $\deg_G(i)\leq 3$.  Since $i$ has at least two neighbors, $\deg_G(i)\geq 2$.  If $\deg_G(i)=2$, then $\ft{k}{i}{j}$.  Suppose $\deg_G(i)=3$.  On the set $N_G(i)$, the white non-edges can form $P_2$, $P_3$, or $C_3$.  In the case of $P_2$ and $P_3$, one of $j$ and $k$ must be the endpoint of the path, say $k$, so $\ft{k}{i}{j}$.  If it is $C_3$, then apply the odd cycle rule $(i\rightarrow C)$.  Since this argument works for every white non-edge, every non-edge can be colored blue.  Hence $\Zsap(G)=0$.

Let $G$ be the Petersen graph.  Then $G$ is a $3$-regular graph with diameter $2$.  Thus, $\Zsap(G)=0$, and $\xi(G)=M(G)$ by Theorem \ref{thm:Zsap}.  It is known \cite{AIM} that $M(G)=5$.
\end{proof}

In \cite{xi}, it is asked if $\xi(G)\leq \xi(G-v)+1$ for every graphs $G$ and every vertex $v$ of $G$.  Theorem \ref{thm:Zsap} answers this question in positive for a large number of graph-vertex pairs.
\begin{cor}
Let $G$ be a graph and $v\in V(G)$.  Suppose $\Zsap(G-v)=0$.  Then $\xi(G)\leq \xi(G-v)+1$.
\end{cor}
\begin{proof}
Since $\Zsap(G-v)=0$, $\xi(G-v)=M(G-v)$ by Theorem \ref{thm:Zsap}.  Therefore,
\[\xi(G)\leq M(G)\leq M(G-v)+1=\xi(G-v)+1,\]
where the inequality $M(G)\leq M(G-v)+1$ is given in \cite{EHHLR}.
\end{proof}

\begin{ex}
\label{ex:platonic}
Let $G$ be one of the tetrahedron $K_4$, cube $Q_3$, octahedron $G_8$, dodecahedron $G_{12}$, or icosahedron $G_{20}$.  Then, $\Zsap(G)=0$.  This is trivial for tetrahedron, since it is a complete graph.  The complement of an octahedron is three disjoint edges, which is a forest, so $\Zsap(G)=0$.  For the other three graphs, pick one vertex $i$ and look at its neighborhood $N_G(i)$.  The induced subgraph of $\ol{G}$ on $N_G(i)$ is either a $3$-cycle or a $5$-cycle.  Thus the odd cycle rule could apply, and every non-edge in $N_G(i)$ is colored blue.  After doing this to every vertex, by picking one vertex and look at its local game, all white non-edge incident to this vertex will be colored blue.  Therefore, $\xi(G)=M(G)$.  

It is known \cite{Yeh} that $M(K_4)=3$ and $M(Q_3)=4$.  Since the octahedron graph is strongly regular, in \cite{AIM} it shows $4\leq M(G_8)$; together with the fact $Z(G_8)\leq 4$, we know $M(G_8)=4$.  For $G_{12}$ and $G_{20}$, the zero forcing numbers can be computed through the computer program and both equal to $6$, but the maximum nullity is not yet known.
\end{ex}

\begin{defn}
Let $G$ be a graph with some non-edges $B_E$ colored blue.  The color change rule for $\Zsap^+$ (CCR-$\Zsap^+$) is the following:
\begin{itemize}
\item Let $\{j,k\}$ be a non-edge.  If $i\rightarrow j$ in $\lg{G}{B_E}{k}{Z_+}$ for some vertex $i$, then the non-edge $\{j,k\}$ is changed to blue.  This is denoted as $\ft{k}{i}{j}$.
\item The odd cycle rule can be used to perform a force.
\end{itemize}
Similarly, the color change rule of $\Zsap^\ell$ (CCR-$\Zsap^\ell$) is defined through the local game $\lg{G}{B_E}{i}{Z_\ell}$.  As usual, $\Zsap^+(G)$ (respectively, $\Zsap^\ell$) is the minimum number of blue non-edges such that every non-edge will become blue by repeated applications of CCR-$\Zsap^+$ (respectively, CCR-$\Zsap^\ell$).
\end{defn}

\begin{obs}
For any graph $G$, $\Zsap^+(G)\leq \Zsap^\ell(G)\leq \Zsap(G)$.
\end{obs}

By a proof analogous to that of Theorem \ref{thm:Zsap}, we can establish Theorem \ref{thm:Zsapl+}.  Observe that $\Zsap^\ell(G)=0$ implies $\Zsap^+(G)=0$.
\begin{thm}
\label{thm:Zsapl+}
Let $G$ be a graph.  If $\Zsap^\ell(G)=0$, then every matrix in $\S_\ell(G)$ has the SAP.  If $\Zsap^+(G)=0$, then every matrix in $\S_+(G)$ has the SAP.  Therefore, if $\Zsap^+(G)=0$, then $M_+(G)=\nu(G)$.
\end{thm}

\begin{cor}
Suppose $G$ is a graph with $\Zsap^+(G)=0$.  Then $\xi(G)\geq M_+(G)$.
\end{cor}

\begin{ex}
Let $G=K_{n_1,n_2,\ldots ,n_p}$ be a complete multi-partite graph with $n_1\geq n_2\geq \cdots \geq n_p$ and $p\geq 2$.  Denote $n=\sum_{t=1}^pn_t$.  Then $\Zsap^\ell(G)=\Zsap^+(G)=0$, so $\nu(G)=M_+(G)=n-n_1$ \cite{HLA46}.  On the other hand, if $n_1\geq 4$, then $\Zsap(G)> 0$, since none of the non-edges in this part can be colored.
\end{ex}

\begin{ex}
If $T$ is a tree, then $\Zsap^+(T)=0$.  However, not every tree $T$ has $\Zsap^\ell(T)=0$.  For example, let $G$ be the graph obtained from $K_{1,4}$ by attaching four leaves to the four existing leaves.  In this graph, only the non-edges incident to the center vertex can be colored by CCR-$\Zsap^\ell$, so $\Zsap^\ell(G)> 0$.

\end{ex}\

\subsection{Graph join}
\label{sec:graphjoin}

Since the SAP zero forcing process uses a propagation on non-edges, it is interesting to consider $\Zsap(G)$ if $\ol{G}$ has two or more components; that is, $G$ is a join of two or more graphs.

\begin{prop}
Let $G$ and $H$ be two graphs.  Then 
\[\Zsap(G\vee H)=\Zsap(G\vee K_1)+\Zsap(H\vee K_1).\]
\end{prop}
\begin{proof}
Let $v$ be the vertex corresponding to the $K_1$ in $G\vee K_1$.  Denote $E_1=E(\ol{G})$ and $E_2=E(\ol{H})$.  Consider the mapping $\pi: V(G\vee H)\rightarrow V(G\vee K_1)$ such that $\pi(i)=i$ if $i\in V(G)$ and $\pi(i)=v$ if $i\in V(H)$.  Fix a vertex $u\in V(H)$, consider the mapping $\pi^{-1}: V(G\vee K_1)\rightarrow V(G\vee H)$ such that $\pi^{-1}(i)=i$ if $i\in V(G)$ and $\pi^{-1}(v)=u$.

Suppose at some stage $B_E$ is the set of blue non-edges in $G\vee H$, and $B_E\cap E_1$ and $B_E\cap E_2$ are the sets of blue non-edges in $G\vee K_1$ and $H\vee K_1$ respectively.  Let $e=\{j,k\}\in E_1$.  If $\ft{k}{i}{j}$ happens in $G\vee H$, then $\ft{k}{\pi(i)}{j}$ can be applied in $G\vee K_1$; if $\ft{k}{i}{j}$ happens in $G\vee K_1$, then $\ft{k}{\pi^{-1}(i)}{j}$ can be applied in $G\vee H$.  Also, if $e$ is in some cycle $C$ and $(i\rightarrow C)$ happens in either $G\vee H$ or $G\vee K_1$, then by the definition of the odd cycle rule $C$ must totally fall in $V(G)$.  If $(i\rightarrow C)$ in $G\vee H$, then $(\pi(i)\rightarrow C)$ in $G\vee K_1$; if $(i\rightarrow C)$ in $G\vee K_1$, then $(\pi^{-1}(i)\rightarrow C)$ in $G\vee H$.  Similarly, all these correspondences work when $e\in E_2$.

Therefore, we can conclude that $B_E$ is a ZFS-$\Zsap$ in $G\vee H$ if and only if $B_E\cap E_1$ and $B_E\cap E_2$ are ZFS-$\Zsap$ in $G\vee K_1$ and $H\vee K_1$ respectively.
\end{proof}

\begin{ex}
\label{ex:highZsap}
The value of $\Zsap(G\vee K_1)$ and the value of $\Zsap(G)$ can vary a lot.  For example, when $G=\ol{K_n}$, we will show that $\Zsap(\ol{K_n})={n\choose 2}$ and $\Zsap(\ol{K_n}\vee K_1)=\Zsap(K_{1,n})={n-1\choose 2}-1$ when $n\geq 3$.

Since there are no edges in $\ol{K_n}$, no vertex can make a force in any local game.  This means $\Zsap(\ol{K_n})={n\choose 2}$.  

For $K_{1,n}$, color some edges $B_E$ of $\ol{K_{1,n}}$ blue so that the set of white non-edges forms a $3$-cycle with $n-3$ leaves attaching to a vertex of the $3$-cycle.  Then $B_E$ is a ZFS-$\Zsap$ for $K_{1,n}$, since the $n-3$ leaves can be colored by forcing triples, and then the $3$-cycle can be colored by the odd cycle rule.  Therefore, $\Zsap(K_{1,n})\leq {n-1\choose 2}-1$.

Conversely, suppose $B_E$ is a ZFS-$\Zsap$ of $K_{1,n}$ with $|B_E|={n-1\choose 2}-2$.  Let $\ol{G}_W$ be the graph whose edges are the white non-edges.  Then $|E(\ol{G}_W)|=n+1$.  Obtain a subgraph $H$ of $\ol{G}_W$ by deleting leaves and isolated vertices repeatedly until there is no leaf left.  Thus $H$ has minimum degree at least two.  Since deleting a leaf removes an edge and a vertex, $|V(H)|+1\leq |E(H)|$.  This means $H$ must contain a component that is not a cycle (so in particular not an odd cycle).  Since this component has minimum degree at least two, none of its edge can be colored, a contradiction.  Hence $\Zsap(K_{1,n})={n-1\choose 2}-1$.
\end{ex}

\begin{prop}
\label{prop:noiso}
For any graph $G$, $\Zsap(G\vee K_1)\leq \Zsap(G)$.  If $G$ contains no isolated vertices, then $\Zsap(G\vee K_1)=\Zsap(G)$.
\end{prop}
\begin{proof}
Every ZFS-$\Zsap$ for $G$ is a ZFS-$\Zsap$ for $G\vee K_1$, so $\Zsap(G\vee K_1)\leq \Zsap(G)$.

Now consider the case that $G$ has no isolated vertices.  Suppose at some stage $B_E$ is the set of blue non-edges for both $G\vee K_1$ and $G$.  We claim that if a non-edge $\{j,k\}\in E(\ol{G})$ is colored in $G\vee K_1$, then it can also be colored in $G$.  

Label the vertex in $V(K_1)$ as $v$.  If $\ft{k}{i}{j}$ in $G\vee K_1$ with $i\neq v$, then it is also a forcing triple in $G$.  Suppose $\ft{k}{v}{j}$ happens in $G\vee K_1$.  Then it must be the case when $j$ is the only white vertex in $\lg{G\vee K_1}{B_E}{k}{Z}$, since $v$ is a vertex that is adjacent to every vertex and it cannot make a force unless every vertex except $j$ is already blue.  Since $j$ is not an isolated vertex, it has a neighbor $i'$ in $V(G)$.  Now $\ft{k}{i'}{j}$ can make $\{j,k\}$ blue.  Therefore, $\Zsap(G\vee K_1)= \Zsap(G)$.

\end{proof}

\begin{prop}
\label{prop:GvK1}
Let $G$ be a graph.  Then $\Zsap(G\vee K_1)=0$ if and only if one of following holds:
\begin{itemize}
\item $G$ has no isolated vertices and $\Zsap(G)=0$.
\item $G=K_1$ or $G$ is a disjoint union of a connected graph $H$ and an isolated vertex such that $\Zsap(H)=0$.
\item $G=\ol{K_3}$.
\end{itemize}
\end{prop}
\begin{proof}
Let $v$ be the vertex in $V(K_1)\subseteq V(G\vee K_1)$.  In the case that $G$ has no isolated vertices, $\Zsap(G\vee K_1)=0$ if and only if $\Zsap(G)=0$ by Proposition \ref{prop:noiso}.  If $G=K_1$, then $\Zsap(K_2)=0$.  If $G=\ol{K_3}$, then $\Zsap(K_{1,3})=0$.  Finally, suppose $G$ is a disjoint union of a connected graph $H$ and an isolated vertex $w$ such that $\Zsap(H)=0$.  Then every forcing triple in $H$ can work in $G\vee K_1$ to make all non-edges in $H$ blue.  After that, $\ft{k}{v}{w}$ takes action in $G\vee K_1$ for every $k\in V(H)$.  Thus, every non-edge in $G\vee K_1$ is blue.  

For the converse statement, suppose $\Zsap(G\vee K_1)=0$ and no initial blue non-edge is given for $G\vee K_1$.  Suppose $G$ has $p$ components with vertex sets $V_1, V_2, \ldots ,V_p$.  Call a non-edge with two endpoints in different components in $G$ as a crossing non-edge.  We claim that if $p\geq 3$, then no crossing non-edge can turn blue in $G\vee K_1$ by any forcing triples.  Let $\{j,k\}$ be a crossing non-edge.  Without loss of generality, let $k\in V_1$ and $j\in V_2$.  Suppose at some stage $B_E$ is the set of blue non-edges and none of the crossing non-edges is blue.  In the local game $\lg{G\vee K_1}{B_E}{k}{Z}$, all blue vertices are contained in $V_1\cup\{v\}$, since all the crossing non-edges are white.   If $\ft{k}{i}{j}$ happens in $G\vee K_1$, it must be the case that $i=v$, since $v$ is the only blue neighbor of $j$ in $\lg{G\vee K_1}{B_E}{k}{Z}$.  Pick a vertex $u\in V_3$.  Since both $j$ and $u$ are white neighbors of $v$ in $\lg{G\vee K_1}{B_E}{k}{Z}$, it is impossible that $\ft{k}{i}{j}$ is a forcing triple.  In conclusion, if $\Zsap(G\vee K_1)=0$ and $G$ contains at least three components, the odd cycle rule must be applied to the crossing non-edges.  Therefore, $G$ must be $\ol{K_3}$ in this case.  

If $G$ has only one component, then $G$ contains no isolated vertices, unless $G=K_1$.  Otherwise assume $G$ has an isolated vertex and has exactly two components.  Then $G$ must be a disjoint union of a connected graph $H$ and an isolated vertex $w$.  Now we build a sequences of forces for $H$ according to the forces in $G\vee K_1$.  Suppose $\ft{k}{i}{j}$ happens in $G\vee K_1$ with $j,k\in V(H)$.  If $i\in V(H)$, then $\ft{k}{i}{j}$ also works in $H$.  If $i\notin V(H)$, then it must be $\ft{k}{v}{j}$.  But $v$ is adjacent to every vertex, so in $\lg{G}{B_E}{k}{Z}$ every vertex except $j$ must be blue.  Since $H$ is connected, there must be a vertex $i'$ that is adjacent to $j$.  Thus, $\ft{k}{i'}{j}$ can make $\{j,k\}$ blue.  Therefore, if $\Zsap(G\vee K_1)=0$, then $\Zsap(H)=0$.
\end{proof}

\subsection{Computational results for small graphs}
\label{sec:simulation}
Table \ref{tbl:Zsap} shows the proportions of graphs that have certain parameters equal to 0, over all connected graphs with a fixed number of vertices.  Graphs are not labeled and isomorphic graphs are considered as the same.  The computation is done by \textit{Sage} and the code can be found in \cite{SageZsap}.

\begin{table}[h]
\begin{center}
    \begin{tabular}{c|ccc}
    $n$ & $\Zsap=0$ & $\Zsap^\ell=0$ & $\Zsap^+=0$ \\ \hline
    1   & 1.0       & 1.0            & 1.0         \\
    2   & 1.0       & 1.0            & 1.0         \\
    3   & 1.0       & 1.0            & 1.0         \\
    4   & 1.0       & 1.0            & 1.0         \\
    5   & 0.86      & 0.95           & 0.95        \\
    6   & 0.79      & 0.92           & 0.92        \\
    7   & 0.74      & 0.89           & 0.89        \\
    8   & 0.73      & 0.88           & 0.88        \\
    9   & 0.76      & 0.89           & 0.89        \\
    10  & 0.79      & 0.90           & 0.91        \\
    \end{tabular}
\caption{The proportion of graphs satisfies $\zeta(G)=0$ over all connected graphs on $n$ vertices.}
\label{tbl:Zsap}
\end{center}
\end{table}

In Section \ref{sec:xi}, we apply these results to help compute the value of $\xi(G)$ when $|G|\leq 7$.

\section{A vertex cover version of the SAP zero forcing game}
\label{sec:Zvc}

As Example \ref{ex:highZsap} points out, for a connected graph $G$ on $n$ vertices, the value of $\Zsap(G)$ can be much higher than $n$.  This section considers a vertex cover version of the SAP zero forcing game.  That is, if $B$ is a set of vertices, then consider the \textit{complementary closure} $\ccl(B)$ as all those non-edges that are incident to any vertex in $B$.  Now instead of picking some non-edges as blue at the beginning, we pick a set of vertices $B$, and color the set $\ccl(B)$ blue initially.  

Following this idea, a new parameter $\Zvc(G)$ is defined with $0\leq \Zvc(G)\leq n$, and Theorem \ref{thm:Zvc} shows that $M(G)-\Zvc(G)\leq \xi(G)$.

\begin{defn}
For a graph $G$, the parameter $\Zvc(G)$ is the minimum number of vertices $B$ such that by coloring $\ccl(B)$ blue, every non-edge will become blue by repeated applications of CCR-$\Zsap$ with the restriction 
\begin{itemize}
\item $\ft{k}{i}{j}$ cannot perform a force if $i\in B$ and $\{i,k\}\in E(\ol{G})$. 
\end{itemize}
A set $B\subseteq V(G)$ with this property is called a $\Zvc$ zero forcing set.
\end{defn}

\begin{thm}
\label{thm:Zvc}
Let $G$ be a graph.  Then 
\[M(G)-\Zvc(G)\leq \xi(G).\]
\end{thm}
\begin{proof}
For given $G$ and $A=\begin{bmatrix}a_{i,j}\end{bmatrix}\in\S(G)$,  let $d=\Zvc(G)$ and $\ol{m}=|E(\ol{G})|$.  Pick an order for the set of non-edges, and let $\Psi$ be the SAP matrix for $A$ with respect to the given order.   Let $B$ be a ZFS-$\Zvc$ with $|B|=d$.  We will show that we can perturb the diagonal entries of $A$ corresponding to $B$ such that the new matrix has the SAP.

Denote $W=E(\ol{G})-\ccl(B)$ as the initial white non-edges.  Since $B$ is a ZFS-$\Zvc$, every non-edge in $W$ is forced to blue at some stage.  Say at stage $t$, $W_t$ is the set of white non-edges that are forced blue.  The set $W_t$ can be one non-edge, or the edges of an odd cycle; thus, $\{W_t\}_{t=1}^s$ forms a partition of $W$, where $s$ is the number of stages it takes to color all non-edges blue.  Define $U_t$ as follows:  If $W_t$ is a non-edge colored by the forcing triple $\ft{k}{i}{j}$, then $U_t=\{(i,k)\}$; if $W_t$ is a cycle colored by an odd cycle rule $(i\rightarrow C)$, then $U_t=\{(i,v)\}_{v\in V(C)}$.  Let $U=\bigcup_{t=1}^sU_t$.  

We first show that $\Psi[U,W]$ is nonsingular.  The proof of Theorem \ref{thm:Zsap} shows that if $W_{t_0}$ is given by the odd cycle rule for some step $t_0$, then $\Psi[U_{t_0},W_{t_0}]$ is nonsingular and $\Psi[U_{t_0},\bigcup_{t={t_0+1}}^s W_t]=O$.  We will see that the same property is also true when $W_{t_0}$ a single non-edge.  Suppose at stage $t_0$, the set of blue non-edges is $B_E$ and $\ft{k}{i}{j}$ applies.  Thus, $U_{t_0}=\{(i,k)\}$ and $W_{t_0}=\{\{j,k\}\}$.  By Definition \ref{defn:system},
\[\Psi[U_{t_0},W_{t_0}]=\begin{bmatrix}\Psi_{(i,k),\{j,k\}}\end{bmatrix}=\begin{bmatrix}a_{i,j}\end{bmatrix},\]
which is nonsingular, since $\{i,j\}$ is an edge.  For any white non-edge $e$ that is not incident to $k$, $\Psi_{(i,k),e}=0$.  If $e=\{j',k\}$ is a white non-edge for some $j'\neq j$, then $j'$ is not a neighbor of $i$, for otherwise $i$ has two white neighbors in $\lg{G}{B_E}{k}{Z}$;  therefore, $\Psi_{(i,k),e}=a_{i,j'}=0$.  By column/row permutations according to $\{W_t\}_{t=1}^d$ and $\{U_t\}_{t=1}^d$ respectively, the $\Psi[U,W]$ becomes a lower triangular block matrix, with every diagonal block nonsingular.  Hence $\Psi[U,W]$ is nonsingular.

Now give the non-edges in $\ccl(B)$ an order.  Following the order, for each non-edge $\{i,j\}$ in $\ccl(B)$, put either $(i,j)$ or $(j,i)$ into another ordered set $U_B$.  Since $\Psi_{(i,j),\{i,j\}}=a_{i,i}$, the diagonal entries of $\Psi[U_B,\ccl(B)]$ are controlled by $a_{i,i}$ for some $i\in B$.

Consider the matrix 
\[\Psi[U\cup U_B,W\cup \ccl(B)]=\begin{bmatrix}
\Psi[U,W] & \Psi[U,\ccl(B)] \\
\Psi[U_B,W] & \Psi[U_B,\ccl(B)]
\end{bmatrix}.\]
We claim that those entry $a_{i,i}$ with $i\in B$ only appear on the diagonal of $\Psi[U_B,\ccl(B)]$.  For each $i\in B$, the only possible occurrence of $a_{i,i}$ is in the case $\Psi_{(i,k),\{i,k\}}=a_{i,i}$ for some vertex $k$ and non-edge $\{i,k\}\in E(\ol{G})$.  Suppose $i\in B$ and $\{i,k\}\in E(\ol{G})$.  Then $\{i,k\}\in\ccl(B)$.  Therefore, $\Psi[U,W]$ and $\Psi[U_B,W]$ does not have this type of $a_{i,i}$ with $i\in B$ involved.  Now it is enough to show $(i,k)\notin U$.  Recall that $U=\bigcup_{t=1}^sU_t$.  At stage $t$, if a forcing triple is applied, then $(i,k)\notin U_t$ since $\ft{k}{i}{j}$ is forbidden for any $j$ by the definition; if the odd cycle rule is applied, then $(i,k)\notin U_t$ since $\{i,k\}\in E(\ol{G})$.  Therefore, $\Psi[U,\ccl(B)]$ contains no such $a_{i,i}$ with $i\in B$, either.

Let $D_B$ be the diagonal matrix indexed by $V(G)$ with the $i,i$-entry $1$ if $i\in B$ and $0$ otherwise.  Consider the matrix $A+xD_B$.  By the discussion above, the SAP matrix of $A+xD_B$ is
\[\Psi[U\cup U_B,W\cup \ccl(B)]=\begin{bmatrix}
\Psi[U,W] & \Psi[U,\ccl(B)] \\
\Psi[U_B,W] & \Psi[U_B,\ccl(B)]+xI
\end{bmatrix}.\]
Since $\Psi[U,W]$ is nonsingular, $\Psi[U\cup U_B,W\cup \ccl(B)]$ is nonsingular when $x$ is large enough.  This means, by changing $d=|B|$ diagonal entries of $A$, the corresponding SAP matrix becomes full-rank.  Therefore,
\[M(G)-\Zvc(G)\leq \nul(A+xD_B)\leq \xi(G).\]
\end{proof}

\begin{rem}
Theorem \ref{thm:Zvc} actually proves that if $B$ is a ZFS-$\Zvc$, then every matrix $A\in \S(G)$ attains the SAP by perturbing those diagonal entries corresponding to $B$.
\end{rem}

In classical graph theory, a vertex cover of a graph $G$ is a set of vertices $S$ such that every edge in $G$ is incident to some vertex in $S$;  that is, $G-S$ contains no edges.  The \textit{vertex cover number} $\beta(G)$ is defined as the minimum cardinality of a vertex cover in the graph $G$.  Corollary \ref{cor:beta} below shows the relation between $M(G)$, $\xi(G)$, and $\beta(G)$.

\begin{cor}
\label{cor:beta}
Let $G$ be a graph.  Then 
\[M(G)-\beta(\ol{G})\leq \xi(G).\]
\end{cor}
\begin{proof}
Let $S$ be a vertex cover of $\ol{G}$.  Then $S$ is a ZFS-$\Zvc$, since every non-edge is blue initially.  Therefore, $\Zvc(G)\leq \beta(G)$ and the desired inequality comes from Theorem \ref{thm:Zvc}.
\end{proof}

\begin{ex}
Let $G=K_3\vee \ol{K_4}$.  Then from data in \cite{small}, $M(G)=Z(G)=5$.  Since $G$ is a subgraph of $K_3\vee P_4$, by minor monotonicity $\xi(G)\leq \xi(K_3\vee P_4)\leq Z(K_3\vee\ol{P_4})\leq 4$.  On the other hand, by picking one of the vertex in $V(K_4)$, it forms a ZFS-$\Zvc$, since the initial white non-edges form a 3-cycle and the odd cycle rule applies.  Thus $\Zvc(G)=1$ and $\xi(G)\geq M(G)-\Zvc(G)=4$.  Therefore, $\xi(G)=4$.

Notice that $G$ contains a $K_4$ minor but not a $K_5$ minor, so we can only say $\xi(G)\geq \xi(K_4)=3$ by considering $K_p$ minors.
\end{ex}

Similarly, we can define $\Zvc^\ell(G)$ by changing CCR-$\Zsap$ to CCR-$\Zsap^\ell$.  Then we have Theorem \ref{thm:Zvcl}.

\begin{thm}
\label{thm:Zvcl}
Let $G$ be a graph.  Then
\[M_+(G)-\Zvc^\ell(G)\leq \nu(G).\]
\end{thm}

\begin{rem}
The proof of Theorem \ref{thm:Zvc} relies on the fact $\Psi[U,W]$ is a lower triangular block matrix.  This is not always true for $Z_+$.  As a vertex can force two or more white vertices under CCR-$Z_+$, the sets $\{U_t\}_{t=1}^s$ might not be mutually disjoint and it is possible that $|U|< |W|$.  Therefore, the same proof does not work for $Z_+$.
\end{rem}

\section{Values of $\xi(G)$ for small graphs}
\label{sec:xi}

Analogous to $M(G)\leq Z(G)$, it is shown in \cite{param} that $\xi(G)\leq \ZFloor(G)$, where $\ZFloor(G)$ is defined through a (conventional) zero forcing game with CCR-$\ZFloor$:
\begin{itemize}
\item CCR-$Z$ can be used to perform a force.  Or if $i$ is blue, $i$ has no white neighbors, and $i$ was not used to make a force yet, then $i$ can pick one white vertex $j$ and force it blue.
\end{itemize}

By using \textit{Sage} and with the help of Theorem \ref{thm:Zsap} and Theorem \ref{thm:Zvc}, we will see that $\ZFloor$ agrees with $\xi(G)$ for graphs up to $7$ vertices.  This result also relies on some other lower bounds.  The \textit{Hadwiger number} $\eta(G)$ is defined as the largest $p$ such that $G$ has a $K_p$ minor.  Since $\xi(G)$ is minor monotone, it is known \cite{param} that when $\eta(G)=p$
\[\xi(G)\geq \xi(K_p)=p-1=\eta(G)-1.\]
The $T_3$-family is a family of $6$ graphs \cite[Fig.~2.1]{HvdH06}.  It is known \cite{HvdH06} that a graph $G$ contains a minor in the $T_3$-family if and only if $\xi(G)\geq 3$.

\begin{lem}
\label{lem:xivalue}
Let $G$ be a connected graph with at most $7$ vertices.  Then at least one of the following is true:
\begin{itemize}
\item $\Zsap(G)=0$, which implies $\xi(G)=M(G)$.
\item $G$ is a tree, which implies $\xi(G)=2$ if $G$ is not a path, and $\xi(G)=1$ otherwise.
\item $\ZFloor(G)=M(G)-\Zvc(G)$, which implies $\xi(G)=\ZFloor(G)$.
\item $\ZFloor(G)=\eta(G)-1$, which implies $\xi(G)=\ZFloor(G)$.
\item $\ZFloor(G)=3$ and $G$ contains a $T_3$-family minor, which implies $\xi(G)=3$.
\end{itemize}
\end{lem}
\begin{proof}
By running a \textit{Sage} program \cite{SageZsap}, one of the five cases will happen.  If $\Zsap(G)=0$, then $\xi(G)=M(G)$ by Theorem \ref{thm:Zsap}.  If $G$ is a tree, then $\xi(G)\leq 2$, and the equality holds only when $G$ is not a path \cite{xi}.  Both $M(G)-\Zvc(G)$ and $\eta(G)-1$ are lower bounds of $\xi(G)$ by Theorem \ref{thm:Zvc} and \cite{param}.  When one of the lower bounds meets with the upper bound $\ZFloor(G)$, $\xi(G)=\ZFloor(G)$.  Finally, if $G$ has a $T_3$-family minor, then $\xi(G)\geq 3$ \cite{HvdH06}.  In this case, $\xi(G)=3$ when $\ZFloor(G)=3$.
\end{proof}

While $\xi(T)\leq 2$ for all tree $T$, the value of $\ZFloor(T)$ can be more than two.  Example A.11.~of \cite{param} gives a tree $T$ with $\ZFloor(T)= 3$; the graph $T$ is shown in Figure \ref{fig:tree3}.  However, $\xi(G)=\ZFloor(G)$ is still true when $G$ is a tree and $|G|\leq 7$.

\begin{figure}[h]
\begin{center}\begin{tikzpicture}
\pgfmathsetmacro{\l}{1}
\begin{scope}[every node/.style={circle,draw=black,fill=black!10}]
\node (v0) at (0,0) {};
\foreach \i/\ang in {1/90,2/210,3/330}{
\coordinate (p\i) at (\ang:\l);
\node (v\i) at (p\i) {}; 
\draw (v0) -- (v\i);
\foreach \j in {1,2,3}{
\pgfmathsetmacro{\k}{int(3*\i+\j)}
\pgfmathsetmacro{\kang}{120*\i+60*\j-150}
\path (p\i)++(\kang:\l) coordinate (p\k);
\node (v\k) at (p\k) {};
\draw (v\i) -- (v\k);
}}
\end{scope}

\end{tikzpicture}\end{center}
\caption{An example of tree $T$ with $\ZFloor(T)>3$.}
\label{fig:tree3}
\end{figure}
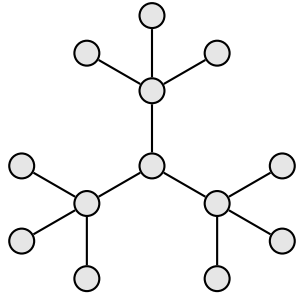

\begin{lem}
\label{lem:ZFtree}
Let $G$ be a tree with at most $7$ vertices.  Then $\xi(G)=\ZFloor(G)$.  
\end{lem}
\begin{proof}
When $G$ is a tree, it is known \cite{xi} that $\xi(G)=2$ when $G$ is not a path, and $\xi(G)=1$ if $G$ is a path.  When $G$ is a path, then $\xi(G)=1=\ZFloor(G)$.  Assume $G$ is not a path.  It is enough to show $\ZFloor(G)\leq 2$.  In this case, $G$ must have a vertex $v$ of degree at least $3$.  Call this type of vertex a high-degree vertex.  If $G$ has only one high degree vertex, then $\ZFloor(G)\leq 2$ since any two leaves form a ZFS-$\ZFloor$.  Since $|G|\leq 7$, there are at most two high-degree vertices.  Pick two leaves such that the unique path between them contains only one high-degree vertex, then these two leaves form a ZFS-$\ZFloor$.
\end{proof}

\begin{thm}
\label{thm:xivalue}
Let $G$ be a graph with at most $7$ vertices.  Then $\xi(G)=\ZFloor(G)$.
\end{thm}
\begin{proof}
Let $G$ be a graph with at most $7$ vertices.  Then $M(G)=Z(G)$ \cite{small}.  If $\Zsap(G)=0$, then $\xi(G)=M(G)=Z(G)$.  Since $\xi(G)\leq \ZFloor(G)\leq Z(G)$, $\xi(G)=\ZFloor(G)$.  If $G$ is a tree, then $\ZFloor(G)=\xi(G)$ by Lemma \ref{lem:ZFtree}.  Then by Lemma \ref{lem:xivalue}, $\xi(G)=\ZFloor(G)$ for all connected graph $G$ up to $7$ vertices.  It is known that $\xi(G_1\dunion G_2)=\max\{\xi(G_1),\xi(G_2)\}$ \cite{xi} and $\ZFloor(G_1\dunion G_2)=\max\{\ZFloor(G_1),\ZFloor(G_2)\}$ \cite{param}, so $\xi(G)=\ZFloor(G)$ for any graph up to $7$ vertices.
\end{proof}

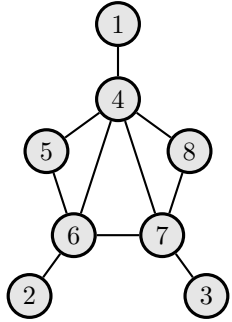
\begin{figure}[h]
\begin{center}\begin{tikzpicture}
\foreach \x/\y in {1/90,2/234,3/306}{
\node[whitenode] (v\x) at (\y:2) {$\x$};
}
\foreach \x/\y in {4/90,5/162,6/234,7/306,8/18}{
\node[whitenode] (v\x) at (\y:1) {$\x$};
}
\draw (v1)--(v4);
\draw (v2)--(v6);
\draw (v3)--(v7);
\draw (v4)--(v5)--(v6)--(v7)--(v8)--(v4);
\draw (v4)--(v6);
\draw (v4)--(v7);
\end{tikzpicture}\end{center}
\caption{A graph $G$ on 8 vertices with $\xi(G)=2$ but $\ZFloor(G)=3$.}
\label{fig:xiZFloor}
\end{figure}

\begin{ex}
Let $G$ be the graph shown in Figure \ref{fig:xiZFloor}.  It is known \cite{JLS} that $M(G)=2$.  Since $G$ is not a disjoint union of paths, $\xi(G)=2$.  Also, it can be computed that $Z(G)=\ZFloor(G)=3$.
\end{ex}

\section{Acknowledgments}
The author thanks Leslie Hogben and Steve Butler for their suggestions.

\bibliography{./JLaTeX/AuthorA,./JLaTeX/JournalA,./JLaTeX/JepBib}{}
\bibliographystyle{plain}

%
\end{document}